\documentclass[12pt,leqno]{amsart}
\usepackage{amsfonts,amsthm,amsmath,comment,xcolor,siunitx}

\theoremstyle{plain}
\newtheorem{thm}{Theorem}[section]
\newtheorem{prop}[thm]{Proposition}
\newtheorem{lem}[thm]{Lemma}

\theoremstyle{definition}
\newtheorem{df}[thm]{Definition}
\newtheorem{rem}[thm]{Remark}
\newtheorem{ex}[thm]{Example}

\usepackage{comment}

\newcommand{\FF}{\mathbb{F}}

\newcommand{\RR}{\mathbb{R}}

\newcommand{\I}{\mathcal{I}}

\newcommand{\Z}{\mathbb{Z}}

\newcommand{\R}{\mathbb{R}}

\def\bm#1{\mathbf{#1}}

\DeclareMathOperator{\supp}{supp}

\DeclareMathOperator{\wt}{wt}

\DeclareMathOperator{\Harm}{Harm}
\DeclareMathOperator{\Hom}{Hom}

\begin{document}

\title{{
Harmonic Tutte Polynomials of Matroids
}
}

\author[Chakraborty]{Himadri Shekhar Chakraborty*}
\thanks{*Corresponding author}
\address
	{
		Department of Mathematics, Shahjalal University of Science and Technology\\ Sylhet-3114, Bangladesh\\
	}
\email{himadri-mat@sust.edu}

\author[Miezaki]{Tsuyoshi Miezaki}
\address
	{
		Faculty of Science and Engineering, 
		Waseda University, 
		Tokyo 169-8555, Japan\\ 
	}
\email{miezaki@waseda.jp} 

\author[Oura]{Manabu Oura}
\address
	{
			Institute of Science and Engineering, 
			Kanazawa University,  
			Ishikawa 920-1192, Japan
	}
\email{oura@se.kanazawa-u.ac.jp} 

\date{}
\maketitle

\begin{abstract}
In the present paper, 
we introduce the concept of harmonic Tutte polynomials of matroids and 
discuss some of their properties. 
In particular, we generalize Greene's theorem,
thereby expressing harmonic weight enumerators of codes as evaluations of 
harmonic Tutte polynomials. 
\end{abstract}

{\small
\vspace{-0.15in}

\quad
}


\section{Introduction}

In 1954, Tutte~\cite{Tutte1954} introduced a celebrated graph polynomial
which is originally called the dichromatic polynomial (also see~\cite{Tutte1967}), now known as the Tutte polynomial,
plays a key role in the study of the graph properties related to counting problems. 
In 1969, Crapo~\cite{Crapo} generalised the Tutte polynomial for graphs with
respect to matroids.
Later, Greene~\cite{Greene1976} proved a remarkable connection between 
the weight enumerator $W_{C}(x,y)$ of an $[n,k]$ code $C$ and the Tutte polynomial of its matroid $M_{C}$; this identity, which is known as Greene identity, is as follows:
\[
W_{C}(x,y) 
=
(x-y)^{k}
y^{n-k}
T\left(M_{C}; \dfrac{x+(q-1)y}{x-y},\dfrac{x}{y}\right).
\]
As an application of the above relation, Greene~\cite{Greene1976}
gave an alternative proof of the MacWilliams identity for the 
weight enumerator of an $[n,k]$ code.

Delsarte~\cite{Delsarte} introduced 
discrete harmonic functions on a finite set.
Bachoc~\cite{Bachoc, BachocNonBinary}
associated the discrete harmonic functions to linear codes by introducing
the 
harmonic weight enumerator $W_{C,f}$ of a linear code~$C$.
Bachoc~\cite{BachocNonBinary} and Tanabe~\cite{Tanabe2001} 
independently both proved 
a MacWilliams identity for the harmonic weight enumerators of linear codes over~$\FF_{q}$.

In this paper, we introduce the notion of the harmonic Tutte polynomials
associated with a harmonic function of a certain degree, 
and give a Greene-type identity which we call the \emph{generalized Greene identity},
which relates the harmonic weight enumerator of a code and the harmonic Tutte polynomial
of the matroid corresponding to the code. 
Moreover, as an application of the generalized Greene identity, 
we give a combinatorial proof of Bachoc and Tanabe's MacWilliams-type identity which is stated in Theorem~\ref{thm: Bachoc iden.}.

This paper is organized as follows:
In Section~\ref{Sec:Preli}, 
we present the basic definitions and properties in coding theory and matroid theory used in this paper.  
In Section~\ref{Sec:HarmGenPoly}, 
we define the harmonic Tutte polynomial,
and obtain a relation between the harmonic Tutte polynomials
of a matroid and its dual (Theorem~\ref{Thm:HarmTutteMacIden.}).
Moreover,
we reinterpret the definition of
harmonic weight enumerators of codes (Theorem~\ref{Thm:reinter}). 
In Section~\ref{Sec:GenGreeneIden},
we give a generalization of Greene identity (Theorem~\ref{thm:Miezaki})
with an application in the proof of the MacWilliams identity for
harmonic weight enumerator. 
Finally, in Section~\ref{sec:rem}, 
we conclude the paper with some remarks. 


\section{Basic definitions and notions}\label{Sec:Preli}

In this section, we give some basic definitions and properties of codes and matroids 
that are necessary for this paper. We follow~\cite{HP2003, MS1977, MOSS2019, Oxley1992, Tanabe2001} for the discussions. Moreover, we recall some definitions and properties of the (discrete) harmonic functions; see~\cite{Bachoc, Delsarte, Tanabe2001} for more detail.

\subsection{Discrete harmonic functions}
 
Let $E := \{1,2,\ldots,n\}$ where $n$ is a positive integer.
Let $2^E$ denote the set of all subsets of $E$.
We define  
$E_{d} := \{ X \in 2^{E} \mid |X| = d\}$
for $d = 0,1, \ldots, n$. 
We denote by 
$\R 2^{E}$ and $\R E_{d}$
the real vector spaces spanned by the elements of  
$2^{E}$ and $E_{d}$,
respectively. 
An element of 
$\R E_{d}$
is denoted by
\begin{equation}\label{Equ:FunREd}
	f :=
	\sum_{Z \in E_{d}}
	f(Z) Z
\end{equation}
with coefficients $f(Z) \in \RR$. 
Thus $\RR E_{d}$ is identified with the 
real-valued function on $E_{d}$ given by 
$Z \mapsto f(Z)$. 
Such an element 
$f \in \R E_{d}$
can be extended to an element 
$\widetilde{f}\in \R 2^{E}$
by setting, for all 
$X \in 2^{E}$,
\begin{equation}\label{Equ:TildeF}
	\widetilde{f}(X)
	:=
	\sum_{Z\in E_{d}, Z\subset X}
	f(Z).
\end{equation}
Note that 
$\widetilde{f}(X) = 0$
for any $X \in 2^{E}$ such that $|X| < d$. 
If an element 
$g \in \R 2^{E}$
is equal to
$\widetilde{f}$ 
for some
$f \in \R E_{d}$, then
we say that $g$ has degree $d$. 
We call the vector space $\R E_{d}$ the \emph{homogeneous space} of degree~$d$,
and denote it by~$\Hom_{d}(n)$.
The differentiation $\gamma$ is the operator on $\RR E_{d}$ defined by 
linearity from the identity 
\begin{equation}\label{Equ:Gamma}
	\gamma(Z) := 
	\sum_{Y\in {E}_{d-1}, Y\subset Z} 
	Y
\end{equation}
for all 
$Z \in E_{d}$
and for all $d=0,1, \ldots n$. 
Also, $\Harm_{d}(n)$ is the kernel of~$\gamma$:
\[
	\Harm_{d}(n)
	:= 
	\ker
	\left(
	\gamma\big|_{\R E_{d}}
	\right).
\]
 
 \begin{rem}[\cite{Bachoc,Delsarte}]\label{Rem:Gamma}
 	Let $f \in \Harm_{d}(n)$. 
 	Then $\gamma^{d-i}(f) = 0$ for all $0 \leq i \leq d-1$.
 	This means from definition~(\ref{Equ:Gamma})
 	that
 	$$\sum_{X \in E_{i}}
 	\left(
 	\sum_{\substack{Z \in E_{d}, X \subseteq Z}} 
 	f(Z)
 	\right) 
 	X = 0.$$
 	This implies that
 	$\sum_{\substack{Z \in E_{d}, X \subset Z}} f(Z) = 0$
 	for any $X \in E_{i}$.
 \end{rem}

\begin{rem}\label{Rem:New}
	Let~$f \in \Harm_{d}(n)$. 
	Since $\sum_{Z \in E_{d}} f(Z) = 0$, 
	then it is easy to check 
	from~(\ref{Equ:Gamma}) that
	$\sum_{X \in E_{t}} \widetilde{f}(X) = 0$, where
	$1 \leq d \leq t \leq n$.
\end{rem}

To make the the above definitions and remarks related to harmonic functions easier to understand, we give the following example.

\begin{ex}\label{Ex:HarmonicFunction}
	Let $E = \{1,2,3\}$ and $d =1$. 
	Let $f \in \RR E_{1}$ be the element
	\[
		f 
		= 
		a \{1\}
		+
		b \{2\}
		+
		c \{3\},
	\]
	where $a = f(\{1\})$, $b = f(\{2\})$ and $c = f(\{3\})$. Then 
	$\gamma(f)= (a + b + c) \emptyset$.
	Suppose that $f \in \Harm_{1}(3)$. 
	This implies $\gamma(f) = 0$. 
	So, $a + b + c = 0$. That is, $c = -(a+b)$. 
	Hence
	\begin{equation*}
		\Harm_{1}(3)
		\ni
		f
		=
		a\{1\}+b\{2\}-(a+b)\{3\}.		
	\end{equation*}
	For $f \in \Harm_{1}(3)$, we have the following $\widetilde{f}$'s by (\ref{Equ:TildeF}):
	\begin{table}[h!]
		\begin{center}
			\begin{tabular}{l l l l}
				$\widetilde{f}(\emptyset) = 0$, & $\widetilde{f}(\{1\}) = a$, & $\widetilde{f}(\{2\}) = b$, & $\widetilde{f}(\{3\}) = -(a+b)$\\
				$\widetilde{f}(\{1,2\}) = a+b $, & $\widetilde{f}(\{1,3\}) = -b$, & $\widetilde{f}(\{2,3\}) = -a$, & $\widetilde{f}(\{1,2,3\}) = 0$.
			\end{tabular}
		\end{center}
		\label{Tab:ftilde}	
	\end{table} 
\end{ex}

\subsection{Linear codes}

Let $\FF_{q}$ be a finite field of order~$q$,
where $q$ is a prime power. 
Then $V:=\FF_{q}^{n}$ denotes the vector space of dimension~$n$ 
with ordinary inner product:
$$\bm{u}\cdot\bm{v} := u_{1}v_{1} + \cdots + u_{n}v_{n}$$
for $\bm{u},\bm{v} \in V$,
where
$\bm{u} = (u_{1},\ldots,u_{n})$ and $\bm{v} = (v_{1},\ldots,v_{n})$.
Let
$\supp(\bm{u}) := \{i\in E \mid u_{i} \neq 0\}$
and
$\wt(\bm{u}) := |\supp(\bm{u})|$
for $\bm{u} \in V$.
Let $V_{d} := \{\bm{u} \in V \mid \wt(\bm{u}) = d\}$. 
An element 
$f \in \R E_{d}$
can be extended to an element 
$\acute{f}\in \R V$
by setting, for all 
$\bm{u} \in V$,
\begin{equation}\label{Equ:AcuteF}
	\acute{f}(\bm{u})
	:=
	\sum_{\substack{\bm{v}\in V_{d},\\ \supp(\bm{v})\subset \supp(\bm{u})}}
	f(\supp(\bm{v})).
\end{equation}

An $\FF_{q}$-\emph{linear code} of length~$n$ is a linear subspace of $V$. 
An $\FF_q$-linear code of length~$n$ with dimension~$k$ 
is called an $[n,k]$ linear code.
Let $C$ be an $\FF_{q}$-linear code. 
We denote by $C^{\perp}$ the \emph{dual code} of $C$, defined as:
\[
	C^{\perp}
	:=
	\{
	\bm{u} \in V
	\mid
	\bm{u} \cdot \bm{v} = 0
	\mbox{ for all }
	\bm{v} \in C
\}.
\] 
The \emph{weight distribution} of $C$ is the sequence 
$\{A_{i}\mid i=0,1, \dots, n \}$, 
where $A_{i}$ is the number of codewords of weight $i$. 
The polynomial
\[
	W_C(x,y) 
	:= 
	\sum_{{\bm{u}}\in C}
	x^{n - \wt({\bm{u}})}
	y^{\wt(\bm{u})}
	=
	\sum^{n}_{i=0} 
	A_{i} x^{n-i} y^{i}
\]
is called the \emph{weight enumerator} of $C$
and satisfies the MacWilliams identity:
\[
	W_{C^{\perp}}(x,y)
	=
	\dfrac{1}{|C|}
	W_{C}(x+(q-1)y,x-y).
\]

Bachoc~\cite{Bachoc} introduced the concept of harmonic weight enumerator for a
binary code which was later defined for an $\FF_{q}$-linear code by Tanabe~\cite{Tanabe2001} as follows.

\begin{df}\label{DefHarmWeightBachoc}
	Let $C$ be an $\FF_{q}$-linear code of length $n$. Let $f\in\Harm_{d}(n)$. 
	The \emph{harmonic weight enumerator} associated with $C$ and $f$ is
	\[
		W_{C,f}(x,y) 
		:=
		\sum_{{\bf u}\in C}
		\acute{f}({\bf u})
		x^{n-\wt({\bf u})}
		y^{\wt({\bf u})}.
	\]
\end{df}

\begin{thm}[\cite{Tanabe2001}, MacWilliams type identity]\label{thm: Bachoc iden.} 
	Let $W_{C,f}(x,y)$ be 
	the harmonic weight enumerator of an $\FF_{q}$-linear code $C$ 
	associated to $f \in \Harm_{d}(n)$. Then 
	\[
		W_{C,f}(x,y) = (xy)^{d} Z_{C,f}(x,y),
	\]
	where $Z_{C,f}$ is a homogeneous polynomial of degree $n-2d$, and satisfies
	\[
		Z_{C^{\perp},f}
		(x,y)
		= 
		(-1)^{d} 
		\frac{q^{n/2}}{|C|} 
		Z_{C,f} 
		\left( 
		\frac{x+(q-1)y}{\sqrt{q}}, 
		\frac{x-y}{\sqrt{q}} 
		\right).
	\]
\end{thm}

Let $C$ be an $\FF_{q}$-linear code of length $n$ and let $f  \in \Harm_{d}(n)$. 
Then the weight distribution of $C$ associated to $f$ is defined as
$$A_{i,f} := \sum_{{\bm{u} \in C},\wt({\bm u}) = i} \acute{f}({\bm u}).$$
Therefore the harmonic weight enumerator of $C$ associated with $f$ can be
rewritten as
	\[
	W_{C,f}(x,y) 
	=
	\sum_{i=0}^{n} 
	A_{i,f} x^{n-i} y^{i}.
	\] 
Now from the above definition we have by Theorem~\ref{thm: Bachoc iden.},
$$Z_{C,f}(x,y) = \sum_{i = 0}^{n} A_{i,f} x^{n-i-d}y^{i-d}.$$

\begin{rem}
	If $\deg f = 0$, then we have $A_{i,f} = A_{i}$, that is, $W_{C,f}(x,y)$ becomes the usual weight enumerator $W_{C}(x,y)$.
\end{rem}

\begin{ex}\label{Ex:HarmWeight}
	Let $C$ be a binary linear code of length~$3$.
	The elements of $C$ are listed as follows:
	\[
		(0,0,0), (0,0,1), (1,1,0), (1,1,1).
	\]
	We consider $f \in \Harm_{1}(3)$ as computed in Example~\ref{Ex:HarmonicFunction}. Then by direct computation, 
	we have the harmonic weight enumerator of~$C$ associated to~$f$ 
	is as follows	
	\[
		W_{C,f} 
		= 
		-(a+b) x^2 y + (a+b) x y^2 = (xy)^{1} Z_{C,f},
	\]
	where $Z_{C,f} = (a+b)(y-x)$.
\end{ex}

\subsection{Matroids}

The matroids can be defined in several equivalent ways.
We prefer the definition which is in terms of independent sets. 
A (finite) \emph{matroid} 
$ M $ is an ordered pair $ (E, \I) $ consisting of set $E$
and a collection $ \I $ of subsets of $E$ 
satisfying the following conditions:  
\begin{itemize}
	\item[(I1)] 
	$ \emptyset \in \I $,
	\item[(I2)] 
	if $ I \in \I $ 
	and 
	$ J \subset I $, 
	then 
	$ J \in \I $, and
	\item[(I3)] 
	if $ I, J \in \I $ 
	with $ |I| < |J| $, 
	then there exists 
	$ j \in J \setminus I $ 
	such that 
	$ I \cup \{ j \} \in \I $.
\end{itemize}

The elements of $ \I $ are called the \emph{independent} 
sets of $ M $, 
and $ E $ is called the \emph{ground set} of $ M $. 
A subset of the ground set $ E $ 
that does not belong to $ \I $ is called \emph{dependent}. 
An independent set is called a \emph{basis} 
if it is not contained in any other independent set. 

It follows from axiom (I3) that the cardinalities of all bases in a
matroid $M$ are equal; this cardinality is called the \emph{rank} of $M$. 
The \emph{rank} $\rho(J)$ of an arbitrary subset $J$ of $E$ 
is the size of the largest independent subset of~$J$.
That is, 
$
	\rho(J) 
	:= 
	\max 
	\{ 
		|I| : 
		I \in \I
		\text{ and }
		I \subset J
	\}.
$
In particular, 
$ \rho(\emptyset) = 0 $.
We call
$\rho(E)$ the rank of~$M$.
We refer the readers to~\cite{Oxley1992} for more information on matroids. 

\begin{df}
	Let~$M$ be a matroid on the set $E$ having a rank function $\rho$.
	The \emph{Tutte polynomial} of $M$ is defined as follows:
	\[
		T(M;x,y)
		:=
		\sum_{J \subset E}
		(x-1)^{\rho(E)-\rho(J)}
		(y-1)^{|J|-\rho(J)}.
	\]
\end{df}

\begin{df}
	Let $A$ be a $k \times n$ matrix over a finite field $\FF_{q}$. 
	This gives a matroid $M[A]$ on the set $E$
	in which a set $I$ is independent if and only if the family of 
	columns of~$A$ whose indices belong to $I$ is linearly independent. 
	Such a matroid is called a \emph{vector matroid}.
\end{df}

For an $\FF_{q}$-linear code $C$, $M_{C}$ denotes the vector matroid that 
corresponds to~$C$. 
In the rest of this note,
we prefer to call $M_{C}$ as a matroid instead of a vector matroid.
Next we recall this construction, which is
treated in~\cite{JP2013}. 
Let $G$ be a $k \times n$ matrix with rank~$k$ 
over the finite field $\FF_{q}$. The set $E$ is
indexing the columns of $G$. Let $\I_{G}$ 
be the family of all subsets~$I$ of $E$
for which the columns of $G$ indexed by $I$
are independent. 
Then $M_{G} := (E, \I_{G})$ is a matroid. 
If
$G_{1}$ and $G_{2}$ are generator matrices of an $\FF_{q}$-linear code $C$, 
then
$(E, \I_{G_1}) = (E, \I_{G_2})$. 
Therefore, the matroid~$M_{C} := (E, \I_{C})$ of an $\FF_{q}$-linear code $C$ is well
defined by $(E, \I_{G})$ for any generator matrix~$G$ of~$C$.

\begin{ex}\label{Ex:Matroid}
	Let $E =\{1,2,3\}$. 
	Let $C$ be a $[3,2]$ code over~$\FF_{2}$
	(given in Example~\ref{Ex:HarmWeight}) 
	with generator matrix as follows:
	\[
	G =
	\begin{bmatrix}
		1 & 1 & 0 \\
		0 & 0 & 1 
	\end{bmatrix}.
	\]
	The columns of the matrix $G$ are indexed by $E$. 
	Now we compute a family $\I$ of subsets~$I \subset E$
	such that the columns of $G$ indexed by $I$ are independent: 
	$$\I = \{\emptyset, \{1\},\{2\},\{3\},\{1,3\},\{2,3\}\}.$$
	Then the matroid corresponding to~$C$ is $M_{C} = (E,\I)$.
\end{ex}

\section{Harmonic generalizations of polynomials}\label{Sec:HarmGenPoly}

\subsection{Harmonic Tutte Polynomials}\label{Sec:HarmTuttePoly}

In this section,
we define the Tutte polynomials of a (finite) matroid $M$ 
associated with a harmonic function. 
We also present a very useful relation between 
the Tutte polynomial of a matroid and its dual 
associated to a harmonic function.

\begin{df}
	Let $M = (E,\I)$ be a matroid with rank function $\rho$,
	and $f \in \Hom_{d}(n)$ be a real-valued function of degree~$d$. 
	Then the \emph{weighted Tutte polynomial} of $M$ 
	associated to~$f$ is defined
	as follows: 
	\[
		T(M,f;x,y)
		:=
		\sum_{J \subset E}
		\widetilde{f}(J)
		(x-1)^{\rho(E)-\rho(J)}
		(y-1)^{|J|-\rho(J)}.
	\] 
	In particular, 
	if $f \in \Harm_{d}(n)$,
	then we call the weighted Tutte polynomial
	$T(M,f;x,y)$ the \emph{harmonic Tutte polynomial}
	associated with~$f$.
\end{df}

We define
\[
	\I^{\ast} 
	:= 
	\{
	I \in 2^{E} 
	\mid 
	I \subset E\setminus A 
	\mbox{ for some } A \in \mathcal{B}(M)
	\},
\]
where
$\mathcal{B}(M)$
be the collection of all bases of $M$.
It is clear by~\cite[Theorem 2.1.1]{Oxley1992}
that $\I^{\ast}$ is the set of independent sets of a matroid on~$E$.
This matroid
$M^{\ast} := (E,\I^{\ast})$ 
is called the \emph{dual matroid}
of $M$.
It is well known that
if~$\rho$ is the rank function of a matroid $M = (E,\I)$, 
then the rank function of $M^{\ast} = (E,\I^{\ast})$
is given as follows: for any $J \subset E$, 
\[
	\rho^{\ast}(J)
	:=
	|J|+\rho(E\setminus J) -\rho(E)
\]
(see~\cite[Proposition 2.1.9]{Oxley1992}).
In particular, 
$\rho^{\ast}(E) + \rho(E) = |E|$.
The correspondence between the harmonic Tutte polynomial of a matroid~$M$ 
and its dual~$M^{\ast}$ associated to a harmonic function
is given as follows:

\begin{thm}\label{Thm:HarmTutteMacIden.}
	Let $M = (E,\I)$ be a matroid with a rank function~$\rho$,
	and let $f \in \Harm_{d}(n)$.
	Then $T(M^{\ast},f;x,y) = (-1)^{d} T(M,f;y,x)$.
\end{thm}

Before giving a proof of the above theorem, 
we need to know about the following technical lemma on harmonic functions from~\cite{Bachoc}. 

\begin{lem}[\cite{Bachoc}]\label{Lem:Bachoc}
	Let $f \in \Harm_{d}(n)$ and $J \subset E$.
	Let
	\[
		f^{(i)}(J)
		:=
		\sum_{\substack{Z \in E_{d},\\ |J \cap Z| = i}}
		f(Z).
	\]
	Then for all 
	$0 \leq i \leq d$,
	$f^{(i)}(J) = (-1)^{d-i} \binom{d}{i} \widetilde{f}(J)$.
\end{lem}

\begin{lem}[\cite{Tanabe2001}]\label{Rem:BachocLem}
	Let $f \in \Harm_{d}(n)$, and let $J \in 2^{E}$. Then
	$\widetilde{f}(J) = (-1)^{d} \widetilde{f}(E\setminus J)$. 
	Furthermore, $\widetilde{f}(J) = 0$ if $|J| > n-d$. 
\end{lem}
\begin{proof}
	Let $I,J \in 2^{E}$ such that $I = E\setminus J$. Then
	\begin{align}\label{Equ:Tanabe}
		\widetilde{f}(J)
		=
		\sum_{\substack{Z \in E_{d},\\Z \subset J}}
		f(Z)
		=
		\sum_{\substack{Z \in E_{d},\\ |Z \cap I|=0}}
		f(Z)
		=
		f^{(0)}(I)
		=
		(-1)^{d} \widetilde{f}(E\setminus J),		
	\end{align}
	which is immediate by Lemma~\ref{Lem:Bachoc}.
	Again, if $|J| > n-d$, then $|I| = |E\setminus J| < d$. 
	Therefore, by (\ref{Equ:TildeF}) and (\ref{Equ:Tanabe})
	we have $\widetilde{f}(J) = 0$.
\end{proof}

\begin{proof}[Proof of Theorem~\ref{Thm:HarmTutteMacIden.}]
	Let $M$ be a matroid on $E$ with rank function $\rho$.
	Then $M^{\ast}$ is the dual matroid of $M$ with
	rank function 
	$\rho^{\ast}(J) = |J|+\rho(E\setminus J)-\rho(E)$
	for any $J \subset E$.
	Therefore,
	\begin{align*}
		T(M^{\ast},f;x,y) 
		& =
		\sum_{J \subset E}
		\widetilde{f}(J)
		(x-1)^{\rho^{\ast}(E)-\rho^{\ast}(J)}
		(y-1)^{|J|-\rho^{\ast}(J)}\\
		& =
		\sum_{J \subset E}
		\widetilde{f}(J)
		(x-1)^{\rho^{\ast}(E)-|J|-\rho(E \setminus J) + \rho(E)}
		(y-1)^{|J|-|J|-\rho(E \setminus J) + \rho(E)}\\
		& =
		\sum_{J \subset E}
		\widetilde{f}(J)
		(x-1)^{|E\setminus J|-\rho(E\setminus J)}
		(y-1)^{\rho(E)-\rho(E\setminus J)}\\
		& =
		(-1)^{d} 
		\sum_{J \subset E}
		\widetilde{f}(E\setminus J)
		(y-1)^{\rho(E)-\rho(E\setminus J)}
		(x-1)^{|E\setminus J|-\rho(E\setminus J)}\\
		& =
		(-1)^{d}
		T(M,f;y,x).
	\end{align*}
	This completes the proof.
\end{proof}

\begin{ex}\label{Ex:HarmTutt}
	We assume that $E = \{1,2,3\}$. 
	From Example~\ref{Ex:HarmonicFunction}
	we have  
	\[
	\Harm_{1}(3)
	\ni
	f
	=
	a \{1\} + b \{2\} - (a+b) \{3\}
	\] 
	is a harmonic function of degree~$1$,
	where $f(\{1\}) = a$, $f(\{2\}) = b$ and $f(\{3\}) = -(a+b)$.
	Let $M_{C} = (E,\I)$ be a matroid corresponding to a $[3,2]$ code~$C$
	that is discussed in Example~\ref{Ex:Matroid}.
	Now we can easily find the ranks of all the subsets of $E$ as follows:
	
	\begin{table}[h!]
		\begin{center}
			\begin{tabular}{l l l l}
				$\rho(\emptyset) = 0$, & $\rho(\{1\}) = 1$, & $\rho(\{2\}) = 1$, & $\rho(\{3\}) = 1$\\
				$\rho(\{1,2\}) = 1 $, & $\rho(\{1,3\}) = 2$, & $\rho(\{2,3\}) = 2$, & $\rho(\{1,2,3\}) = 2$.
			\end{tabular}
		\end{center}
		\label{Tab:rank}	
	\end{table} 
	
	By direct calculation, we have
	\begin{align*}
		T(M_{C},f;x,y)
		& =
		\sum_{J \subset E}
		\widetilde{f}(J)
		(x-1)^{\rho(E)-\rho(J)}
		(y-1)^{|J|-\rho(J)}\\
		& =
		(a+b) (x-1) (y-1)
		- (a+b).
	\end{align*}
\end{ex}

\subsection{Harmonic Weight Enumerator}\label{Sec:HarmWeightEnum}

In this section, 
we introduce a new approach to define the harmonic weight enumerators of an $\FF_{q}$-linear code. 
This formulation is inspired by~Jurrius and Pellikaan~\cite{JP2013}.

\begin{df}
	Let $E$ be a finite set of cardinality~$n$. 
	Again let $C$ be an $\FF_{q}$-linear code of length~$n$.
	Then for an arbitrary subset $J \subset E$,
	we define
	\begin{align*}
		C(J) 
		& := 
		\{
			\bm{c} \in C 
			\mid 
			c_{j} = 0 \text{ for all } j \in J
		\}\\
		\ell(J) 
		& := \dim C(J)\\
		B_{J} 
		& := 
		q^{\ell(J)}-1. 
	\end{align*}
\end{df}

\begin{lem}[\cite{JP2013}]\label{LemRank}
	Let $C$ be an $[n,k]$ linear code with generator matrix~$G$. 
	Assume that the columns of $G$ are indexed by the set $E$.
	Let $G_{J}$ be the $k \times t$ submatrix of $G$ consisting of the columns of $G$ indexed by $J \in E_{t}$, 
	and let $\rho(J)$ be the rank of $G_{J}$. 
	Then $\ell(J) = k - \rho(J)$.
\end{lem}

Now we have the following proposition.

\begin{prop}\label{Prop:Connection}
	Let $f \in \Harm_{d}(n)$ and $J \subset E$.
	Define
	\[
		B_{t,f} 
		:= 
		\sum_{J \in E_{t}} 
		\widetilde{f}(J)B_{J}.
	\]
	Then
	we have the following relation between $B_{t,f}$ and $A_{i,f}$ as follows:
	\[
		B_{t,f} 
		= 
		(-1)^{d}
		\sum_{i=d}^{n-t} 
		\binom{n-d-i}{t-d} A_{i,f},	
	\]
	if $d \leq t \leq n-d$;
	otherwise $B_{t,f} = 0$.
\end{prop}

\begin{proof}
	It is immediate from (\ref{Equ:TildeF}) 
	and Lemma~\ref{Rem:BachocLem} that
	$B_{t,f} = 0$ for $0 \leq t < d$ and~$n-d < t \leq n$.
	We now focus on $t$ with 
	$d \leq t \leq n-d$.
	By Lemma~\ref{Rem:BachocLem},
	\begin{align*}
		B_{t,f} 
		& = 
		\sum_{J \in E_{t}} 
		\widetilde{f}(J)B_{J}\\
		& =
		(-1)^{d}
		\sum_{J\in E_{t}}
		\widetilde{f}(E\setminus J) B_{J}\\
		& =
		(-1)^{d}
		{\binom{t}{d}}^{-1}
		\sum_{\substack{J \in E_{t}, X \in E_{d},\\ X \subset J}} 	\widetilde{f}(E\setminus J)B_{J}.
	\end{align*}
	Therefore, it is sufficient to show that for $d \leq t \leq n-d$,
	\[
		{\binom{t}{d}}^{-1}
		\sum_{\substack{J \in E_{t}, X \in E_{d},\\ X \subset J}} 
		\widetilde{f}(E\setminus J)B_{J}
		=
		\sum_{i=d}^{n-t} 
		\binom{n-d-i}{t-d} A_{i,f}.
	\] 
	Now following the definition of $B_{J}$, we can easily observe that
	\begin{align*}
		\sum_{\substack{J \in E_{t}, X \in E_{d},\\ X \subset J}} 
		\widetilde{f}(E\setminus J)
		B_{J}
		& =
		\sum_{\substack{J \in E_{t}, X \in E_{d},\\ X \subset J}} 
		\sum_{\substack{\bm{c} \in C,\\ \supp(\bm{c}) \cap J = \emptyset,\\ \bm{c} \neq \bm{0}}}
		\widetilde{f}(E\setminus J)\\
		& =
		\sum_{\substack{\bm{c} \in C, \\ \bm{c} \neq \bm{0}}}
		\sum_{\substack{J \in E_{t}, X \in E_{d},\\ 
				\supp(\bm{c})\cap J = \emptyset,\\ X \subset J}}
		\widetilde{f}(E\setminus J)\\
		& =
		\sum_{w = 1}^{n-t}
		\sum_{\substack{\bm{c} \in C, \\ \wt(\bm{c}) = w}}
		\sum_{\substack{J \in E_{t}, X \in E_{d},\\ 
				\supp(\bm{c})\cap J = \emptyset,\\ X \subset J}}
		\widetilde{f}(E\setminus J).
	\end{align*}
	For $X \in E_{d}$ and $\bm{c} \in C$, let
	$$A_{X}(\bm{c}) 
		:= 
		\{
			J \in E_{t} 
			\mid 
			X \subset J 
			\mbox{ and } 
			\supp(\bm{c}) \cap J 
			= 
			\emptyset
		\}.$$
	Then from (\ref{Equ:TildeF}) we have
	\begin{align*}
		\sum_{\substack{J \in E_{t}, X \in E_{d},\\ 
				\supp(\bm{c})\cap J = \emptyset,\\ X \subset J}}
		\widetilde{f}(E\setminus J)
		& =
		\sum_{\substack{X \in E_{d},\\ \supp(\bm{c})\cap X = \emptyset}}
		\sum_{J\in A_{X}(\bm{c})}
		\widetilde{f}(E\setminus J)\\
		& =
		\sum_{\substack{X \in E_{d},\\ \supp(\bm{c}) \cap X = \emptyset}}
		\sum_{J\in A_{X}(\bm{c})}
		\sum_{\substack{Z \in E_{d},\\ Z \subset E\setminus J}}
		f(Z)\\
		& =
		\sum_{\substack{T \in E_{t},\\\supp(\bm{c}) \cap T = \emptyset}}
		\sum_{\substack{X \in E_{d},\\ X \subset T}}
		\sum_{J\in A_{X}(\bm{c})}
		\sum_{\substack{Z \in E_{d},\\ Z \subset E\setminus T}}
		f(Z)\\
		& =
		\binom{t}{d}
		\binom{n-d-w}{t-d}
		\sum_{\substack{T \in E_{t},\\ \supp(\bm{c}) \cap T = \emptyset}}
		\sum_{\substack{Z \in E_{d},\\ Z \subset E\setminus T}}
		f(Z).
	\end{align*}
	For $\bm{c} \in C$ with $\wt(\bm{c}) = w$ such that $0< w < d$, 
	by Remark~\ref{Rem:Gamma}, 
	we have 
	\begin{align*}
	& \sum_{\substack{T \in E_{t},\\T \cap \supp(\bm{c}) = \emptyset}}
	\sum_{\substack{Z \in E_{d},\\ Z \subset E\setminus T}}
	f(Z)\\
	& =
	\sum_{i=0}^{w}
	\sum_{\substack{Y\in E_{i},\\ Y \subset \supp(\bm{c})}}
	\binom{w}{i}
	\binom{n-t-w}{d-i}
	\sum_{\substack{Z \in E_{d},\\ Y \subset Z}}
	f(Z)\\
	& =
	0.
	\end{align*}
	Now for $\bm{c} \in C$ with $\wt(\bm{c}) = w$ such that $d \leq w \leq n-t$, 
	by Remark~\ref{Rem:Gamma}, 
	we have
	\begin{align*}
		&\sum_{\substack{T \in E_{t},\\T \cap \supp(\bm{c}) = \emptyset}}
		\sum_{\substack{Z \in E_{d},\\ Z \subset E\setminus T}}
		f(Z)\\
		& =
		\sum_{\substack{Z \in E_{d},\\ Z \subset \supp(\bm{c})}}
		f(Z)
		+
		\sum_{i=0}^{d-1}
		\sum_{\substack{Y\in E_{i},\\ Y \subset \supp(\bm{c})}}
		\binom{w}{i}
		\binom{n-t-w}{d-i}
		\sum_{\substack{Z \in E_{d},\\ Y \subset Z}}
		f(Z)\\
		& = 
		\sum_{\substack{Z \in E_{d},\\ Z \subset \supp(\bm{c})}}
		f(Z)
		+ 
		0\\
		& =
		\acute{f}(\bm{c})	
	\end{align*}
	Therefore, 
	\begin{align*}
		\sum_{\substack{J \in E_{t}, X \in E_{d},\\ X \subset J}} 
		\widetilde{f}(E\setminus J)
		B_{J}
		& =
		\sum_{w = d}^{n-t}
		\sum_{\substack{\bm{c} \in C,\\ \wt(\bm{c}) = w}}
		\binom{t}{d}
		\binom{n-d-w}{t-d}
		\acute{f}(\bm{c})\\
		& =
		\sum_{w = d}^{n-t}
		\binom{t}{d}
		\binom{n-d-w}{t-d}
		A_{i,f}.
	\end{align*}
This completes the proof.
\end{proof}

\begin{ex}\label{Ex:Proposition}
	Let $C$ be a binary linear code of length~$3$ given in 
	Example~\ref{Ex:HarmWeight}. 
	We consider $f \in \Harm_{1}(3)$ and its $\widetilde{f}$'s
	as in Example~\ref{Ex:HarmonicFunction}. 
	Now we compute $B_{t,f}$ for $1 \leq t \leq 2$ in two ways: 
	by direct computation and by using Proposition~\ref{Prop:Connection}.
	
	\begin{equation*}
		\begin{aligned}[t] 
			B_{1,f}
			& =
			\sum_{J\in E_{1}} 
			\widetilde{f}(J)
			B_{J}\\
			& =
			a.1 + b.1 - (a + b).1\\
			& =
			0
		\end{aligned}\qquad
		\begin{aligned}[t]
			B_{1,f} & = 
			(-1)^{1}
			\sum_{i = 1}^{2}
			\binom{2-i}{0}
			A_{i,f}\\
			& =
			- 
			A_{1,f}
			- 
			A_{2,f}\\
			& =
			(a + b) 
			-(a + b)\\
			& =
			0		
		\end{aligned}
	\end{equation*}
	\begin{equation*}
		\begin{aligned}[t]
			B_{2,f}
			& =
			\sum_{J\in E_{2}} 
			\widetilde{f}(J)
			B_{J}\\
			& =
			(a + b).1 -b.0 - a.0\\
			& =
			a+b
		\end{aligned}\qquad
		\begin{aligned}[t]
			B_{2,f}
			& = 
			(-1)^{1}
			\sum_{i = 1}^{1}
			\binom{2-i}{1}
			A_{i,f}\\
			& =
			- 
			A_{1,f}\\
			& =
			a+b		
		\end{aligned} 
	\end{equation*}
\end{ex}

Now we have the following result.

\begin{thm}\label{Thm:reinter}
	Let $C$ be an $\FF_{q}$-linear code of length~$n$, and let $f \in \Harm_{d}(n)$.
	Then
	\[
		Z_{C,f}(x,y)
		= 
		(-1)^{d}
		\sum_{t=d}^{n-d} 
		B_{t,f} 
		(x-y)^{t-d} y^{n-t-d}.
	\]
\end{thm}

\begin{proof}
	By using Proposition~\ref{Prop:Connection} and using the binomial expansion of 
	$x^{n-i} = ((x-y)+y)^{n-i}$
	we have
	\begin{align*}
		(-1)^{d}
		&
		\sum_{t=d}^{n-d} 
		B_{t,f} 
		(x-y)^{t-d} y^{n-t-d}\\ 
		= & 
		\sum_{t=d}^{n-d} 
		\sum_{i = d}^{n-t} 
		\binom{n-d-i}{t-d}
		A_{i,f} 
		(x-y)^{t-d} y^{n-t-d}\\
		= & 
		\sum_{i=d}^{n-d} A_{i,f} 
		\left( 
			\sum_{t=d}^{n-i} 
			\binom{n-d-i}{t-d} 
			(x-y)^{t-d} y^{(n-d-i)-(t-d)}
		\right) 
		y^{i-d}\\
		= & 
		\sum_{i = d}^{n-d} 
		A_{i,f} 
		x^{n-d-i} y^{i-d}\\
		= &  
		\sum_{i = 0}^{n} 
		A_{i,f} 
		x^{n-i-d} y^{i-d}\\
		= & 
		Z_{C,f}(x,y),
	\end{align*}
	since $A_{i,f} = 0$ for $i<d$ and $i > n-d$.
\end{proof}

\section{Generalized Greene's Identity}\label{Sec:GenGreeneIden}

Let $M_{C}$ be a matroid associated to an $\FF_{q}$-linear code~$C$
of length~$n$.
It is immediate from~\cite{Cameron,Greene1976,JP2013} 
that $(M_{C})^{\ast} = M_{C^{\perp}}$. 
Now we have the following proposition. 
\begin{prop}\label{PropTutteMC}
	Let $C$ be an $[n,k]$ code 
	and $M_{C}$ be its matroid. Let $f$ be a harmonic function of degree $d$. Then 
	\[
		T(M_{C},f;x,y) 
		= 
		\sum_{t=d}^{n-d}\sum_{J \in E_{t}} 
		\tilde{f}(J) 
		(x-1)^{\ell(J)} 
		(y-1)^{\ell(J)-(k-t)}.
	\]  
\end{prop}

\begin{proof}
	The proposition follows from 
	$\ell(J) = k - \rho(J)$ 
	for any $J \in E_{t}$ by Lemma~\ref{LemRank},
	and
	$\rho(E) = k$.
\end{proof}

Now we have the following harmonic generalization of the Greene's identity.

\begin{thm}\label{thm:Miezaki}
	Let $C$ be an $[n,k]$ code  
	and $f$ be a harmonic function with degree $d$.
	Then
	\[
		Z_{C,f}(x,y)
		= 
		(-1)^{d}
		(x-y)^{k-d} y^{n-k-d} 
		T
		\left(
			M_{C},f; 
			\dfrac{x+(q-1)y}{x-y},
			\dfrac{x}{y}
		\right).
	\]
\end{thm}

\begin{proof}
	By using the Proposition~\ref{PropTutteMC} and Remark~\ref{Rem:New}, we can write
	\begin{align*}
	T
	\left(
		M_{C},f; 
		\dfrac{x+(q-1)y}{x-y},
		\dfrac{x}{y}
	\right)
	& = 
	\sum_{t=d}^{n-d}\sum_{J \in E_{t}} 
	\widetilde{f}(J)
	\left(
		\dfrac{qy}{x-y}
	\right)^{\ell(J)} 
	\left(
		\dfrac{x-y}{y}
	\right)^{\ell(J)-(k-t)}\\
	& = 
	\sum_{t=d}^{n-d}\sum_{J \in E_{t}} 
	\widetilde{f}(J)q^{\ell(J)}
	\left(
		\dfrac{y}{x-y}
	\right)^{\ell(J)} 
	\left(
		\dfrac{x-y}{y}
	\right)^{\ell(J)-(k-t)}\\
	& = 
	\sum_{t=d}^{n-d}
	\sum_{J \in E_{t}} 
	\widetilde{f}(J)((q^{\ell(J)}-1)+1)
	(x-y)^{-(k-t)} y^{k-t}\\
	& = 
	\sum_{t=d}^{n-d}
	\sum_{J \in E_{t}} 
	\widetilde{f}(J)(B_{J}+1)
	(x-y)^{-(k-t)} y^{k-t}\\
	& = 
	\sum_{t=d}^{n-d}
	\left(
		\sum_{J \in E_{t}} 
		\widetilde{f}(J)B_{J} 
		+
		\sum_{J \in E_{t}} 
		\widetilde{f}(J)
	\right)
	(x-y)^{-(k-t)} y^{k-t}\\
	& = 
	\sum_{t=d}^{n-d}
	\left(
		B_{t,f}
		+
		0
	\right)
	(x-y)^{-(k-t)} y^{k-t}\\
	& = 
	\sum_{t=d}^{n-d}
	B_{t,f}
	(x-y)^{-(k-t)} y^{k-t}.
\end{align*}

Therefore, from Theorem~\ref{Thm:reinter} we have
\begin{align*} 
	(-1)^{d}
	(x-y)^{k-d} y^{n-k-d} 
	T
	&
	\left(
		M_{C},f; 
		\dfrac{x+(q-1)y}{x-y},
		\dfrac{x}{y}
	\right)\\
	= & 
	(-1)^{d}
	\sum_{t=d}^{n-d}
	B_{t,f}
	(x-y)^{t-d} y^{n-t-d}\\
	= & 
	Z_{C,f}(x,y).
	\end{align*}
This completes the proof.
\end{proof}

Now we give an alternative proof of the $\FF_{q}$-analogue of Bachoc's
MacWilliams type identity (see~\cite{Bachoc}) stated in Theorem~\ref{thm: Bachoc iden.}
as an application of Theorem~\ref{thm:Miezaki}.

\begin{proof}[Proof of Theorem~\ref{thm: Bachoc iden.}]
	Let 
	$C$ be an $[n,k]$ code,
	and $M_{C}$ be its matroid.
	Then
	\begin{align*}
		(-1)^{d} 
		&
		\frac{q^{n/2}}{|C|} 
		Z_{C,f} 
		\left( 
			\frac{x+(q-1)y}{\sqrt{q}}, 
			\frac{x-y}{\sqrt{q}} 
		\right)\\
		& =
		(-1)^{2d}
		\frac{q^{n/2}}{q^{k}}
		\left(\dfrac{qy}{\sqrt{q}}\right)^{k-d}
		\left(\dfrac{x-y}{\sqrt{q}}\right)^{n-k-d}
		T\left(M_{C},f;\dfrac{x}{y},\dfrac{x+(q-1)y}{x-y}\right)\\
		& =
		(x-y)^{n-k-d}
		y^{k-d}
		T
		\left(
			M_{C},f;
			\dfrac{x}{y},
			\dfrac{x+(q-1)y}{x-y}
		\right)\\
		& =
		(-1)^{d}
		(x-y)^{(n-k)-d}
		y^{n-(n-k)-d}
		T
		\left(
			M_{C^{\perp}},f;
			\dfrac{x+(q-1)y}{x-y},
			\dfrac{x}{y}
		\right)\\
		& =
		(-1)^{d}
		(x-y)^{\dim C^{\perp}-d}
		y^{n-\dim C^{\perp}-d}
		T
		\left(
			M_{C^{\perp}},f;
			\dfrac{x+(q-1)y}{x-y},
			\dfrac{x}{y}
			\right)\\
		& =
		Z_{C^{\perp},f}(x,y).
	\end{align*}
Hence Theorem is proved. 
\end{proof}

\section{Concluding Remarks}\label{sec:rem}

Let 
$n,k,t$ and $\lambda$ be non-negative integers 
such that
$n \geq k \geq t$ and~$\lambda \geq 1$.
A $t$-$(n,k,\lambda)$~\emph{design} 
(in short, $t$-design)
is a pair 
$\mathcal{D} := (E,\mathcal{B})$,
where $E$ is a finite set of point of cardinality~$n$,
and $\mathcal{B}$ is a collection of $k$-element
subsets of $E$ called \emph{blocks},
with the property that any~$t$ points are contained 
in precisely $\lambda$ blocks.
Some properties of combinatorial $t$-designs obtained from codes 
were discussed 
in~\cite{{AM1969},{Bachoc},{Bannai-Koike-Shinohara-Tagami},{BS2008},{BRS2009},
	{extremal design H-M-N},{MMN},{extremal design2 M-N},{MN-tec},{RCW1975}} and their analogies in the theory of lattices and vertex operator algebras were discussed 
in~\cite{{Bannai-Koike-Shinohara-Tagami},{BM1},{BM2},{BMY},{Miezaki},{Miezaki2},{MMN}}.


The harmonic functions have many applications; 
particularly, the relations between design theory and 
coding theory were stated in Bachoc~\cite{Bachoc}:
the set of words with fixed weight in a binary code $C$ 
forms a $t$-design if and only if 
$W_{C, f}(x, y) = 0$ for all
$f \in \Harm_{d}(n)$, $1 \leq d \leq t$.
Now we have the following design theoretical remark that gives
an application of the harmonic function connecting
the $t$-designs with matroids.
	
\begin{rem}
	If $T(M_{C},f;x,y) = 0$ for all
	$f \in \Harm_{d}(n)$, $1 \leq d \leq t$,
	then the set of words with fixed weight in an $\FF_{q}$-linear code $C$ 
	forms a $t$-design.
\end{rem}

We will more precisely discuss about a 
relation between matroids and combinatorial designs with respect to ``harmonic Tutte polynomials" in a forthcoming paper. Moreover, we give the generalizations of the results
in this paper in~\cite{BCIM20xx}.

\section*{Acknowledgements}
The authors thank Thomas Britz
for helpful discussions. 
The authors would also like to thank the anonymous reviewers 
for their beneficial comments on 
an earlier version of the manuscript. 
The second author is supported by JSPS KAKENHI (22K03277).

\section*{Data availability statement}
The data that support the findings of this study are available from
the corresponding author.

\end{document}